\documentclass[11pt]{amsart}

\usepackage{amssymb}
\usepackage{latexsym}
\usepackage{amsmath}
\usepackage{amsthm}
\usepackage{amsfonts}
\usepackage{color}
\usepackage{pdfsync}
\usepackage[usenames,dvipsnames]{pstricks}
\usepackage{epsfig}
\usepackage{pst-grad} 
\usepackage{pst-plot} 

\newcommand{\R}{\mathbb{R}}

\numberwithin{equation}{section}

\usepackage{amssymb}

\newcommand{\beq}{\begin{equation} }
\newcommand{\eqq}{\end{equation} }
\newcommand{\cuad}{{\sqcap\kern-.68em\sqcup}}

\newcommand{\equ}[1]{(\ref{#1})}

\def\beeq{\begin{equation}}
\def\eeq{\end{equation}}
\newcommand{\begeqaet}{\begin{eqnarray*}}
\newcommand{\eneqaet}{\end{eqnarray*}}

\let\Section=\section
\def\section{\setcounter{equation}{0}\Section}
\newtheorem{Lem}{Lemma}[section]
\newtheorem{Thm}{Theorem}[section]

\newtheorem{Prop}{Proposition}[section]

\newtheorem{Cor}{Corollary}[section]

\begin{document}
\begin{center}{\bf\Large Existence and concentration of solution for a non-local regional Schr\"odinger equation with competing potentials  }\medskip

\bigskip

\bigskip

{Claudianor O. Alves}

Universidade Federal de Campina Grande\\
Unidade Acad\^emica de Matem\'atica\\
CEP: 58429-900 - Campina Grande - PB, Brazil\\
{\sl coalves@mat.ufcg.edu.br}

\noindent 

{C\'esar E. Torres Ledesma}

Departamento de Matem\'aticas, \\
Universidad Nacional de Trujillo,\\
Av. Juan Pablo II s/n. Trujillo-Per\'u\\
{\sl  ctl\_576@yahoo.es}


\medskip

\medskip
\medskip
\medskip
\medskip

\end{center}

\centerline{\bf Abstract}
In this paper, we study the existence and concentration phenomena of solutions for the following non-local regional Schr\"odinger equation
$$
\left\{
\begin{array}{l}
\epsilon^{2\alpha}(-\Delta)_\rho^{\alpha} u + Q(x)u = K(x)|u|^{p-1}u,\;\;\mbox{in}\;\; \mathbb{R}^n,\\
u\in H^{\alpha}(\mathbb{R}^n)
\end{array}
\right.
$$
where $\epsilon$ is a positive parameter, $0< \alpha < 1$, $1<p<\frac{n+2\alpha}{n-2\alpha}$, $n>2\alpha$; $(-\Delta)_{\rho}^{\alpha}$ is a variational version of the regional fractional Laplacian, whose range of scope is a ball with radius $\rho (x)>0$, $\rho, Q, K$ are competing functions. We study the existence of ground state and we analyze the behavior of semi-classical solutions as $\epsilon \to 0$.  
\medskip


\medskip

\date{}

\setcounter{equation}{0}
\section{Introduction}
The aim of this article is to study the non-linear Schr\"odinger equation with non-local regional diffusion and competing potentials 
$$
\left\{
\begin{aligned}
\epsilon^{2\alpha}(-\Delta)_\rho^{\alpha} u + Q(x)u& = K(x)|u|^{p-1}u,\;\;\mbox{in}\;\; \mathbb{R}^n,\\
u&\in H^{\alpha}(\mathbb{R}^n),
\end{aligned}
\right.
\leqno{(P)}
$$
where $0< \alpha < 1$, $\epsilon >0$, $n> 2\alpha$, $Q,K \in C(\mathbb{R}^n, \mathbb{R}^+)$ are bounded and  the operator $(-\Delta)_{\rho}^{\alpha}$ is a variational version of the non-local regional fractional Laplacian, with range of scope determined by a positive function $\rho \in C(\mathbb{R}^n, \mathbb{R}^+)$, which is defined as
$$
\int_{\mathbb{R}^n}(-\Delta)_{\rho}^{\alpha}u(x)\varphi (x)dx = \int_{\mathbb{R}^n}\int_{B(0, \rho (x))} \frac{[u(x+z) - u(z)][\varphi (x+z) - \varphi (x)]}{|z|^{n+2\alpha}}dzdx
$$ 

In what follows, we will work with the problem
$$
(-\Delta)_{\rho_{\epsilon}}^{\alpha}v + Q(\epsilon x)v = K(\epsilon x)|v|^{p-1}v,\quad x\in \mathbb{R}^n, \leqno({P'})
$$
with and $\rho_\epsilon = \frac{1}{\epsilon}\rho (\epsilon x)$, which is equivalent to $(P)$ by considering the change variable  $v(x) = u(\epsilon x)$ . 

Associated with $(P')$ we have the energy functional $I_{\rho_\epsilon} :X^{\epsilon} \to \mathbb{R}$ defined as
$$
\begin{array}{l}
\displaystyle I_{\rho_{\epsilon}} (v) = \frac{1}{2}\left( \int_{\mathbb{R}^n}\int_{B(0, \frac{1}{\epsilon}\rho (\epsilon x))}\hspace{-.5cm} \frac{|v(x+z) - v(x)|^2}{|z|^{n+2\alpha}} + \int_{\mathbb{R}^n}Q(\epsilon x)|v(x)|^2dx \right) -\\ \mbox{}\\
\;\;\;\;\;\;\;\;\;\;\;\;\;\;\;\displaystyle \frac{1}{p+1}\int_{\mathbb{R}^n} K(\epsilon x)|v(x)|^{p+1}dx,
\end{array}
$$
where $X^{\epsilon}$ denotes the Hilbert space $H^{\alpha}(\mathbb{R}^n)$ endowed with the norm
\begin{equation}\label{08}
\|v\|_{\rho_\epsilon}=\left(\int_{\mathbb{R}^{n}}\int_{B(0,\rho_\epsilon (x))}\frac{|v(x+z) - v(x)|^{2}}{|z|^{n+2\alpha}}dzdx + \int_{\mathbb{R}^{n}}Q(\epsilon x)|v(x)|^{2}dx \right)^{\frac{1}{2}}.
\end{equation}
Hereafter, we say that $v \in X^{\epsilon}$ is a weak solution of $(P')$ if $v$ is a critical point of $I_{\rho_\epsilon}$. In Section 2, Proposition \ref{FSprop1}, it is proved that $\|\,\,\,\|_{\rho_\epsilon}$ is equivalent to the usual norm in $H^{\alpha}(\R^n)$.

Recently, the study on problems of fractional Schr\"odinger equations has attracted much attention from many mathematicians. In the case of the fractional Laplacian $(-\Delta)^\alpha$, Chen \cite{MC} studied the existence of ground sate solution of nonlinear fractional Schr\"odinger equation 
\begin{equation}\label{02}
(-\Delta)^\alpha u + V(x)u = u^p\;\;\mbox{in}\;\;\mathbb{R}^n
\end{equation}  
with unbounded potential. The existence of a ground state of (\ref{02}) is obtained by a Lagrange multiplier method and the Nehari manifold method is used to obtain standing waves with prescribed frequency. If $V(x) = 1$, Dipierro et al. \cite{SDGPEV} proved existence and symmetry results for the solution of equation (\ref{02}). Felmer et al. \cite{PFAQJT}, studied the same equation with a more general nonlinearity $f(x,u)$, they obtained the existence, regularity and qualitative properties of ground states. Secchi \cite{SS} obtained positive solutions of a more general fractional Schr\"odinger equation by the variational method.    

On the other hand, research has been done in recent years regarding regional fractional Laplacian, where the scope of the operator is restricted to a variable region near each point. We mention the work by Guan \cite{guan1} and Guan and Ma \cite{guan2} where they study these operators, their relation with stochastic processes and they develop integration by parts formula, and the work by Ishii and Nakamura \cite{HIGN}, where the authors studied the Dirichlet problem for regional fractional Laplacian modeled on the p-Laplacian.

Recently, Felmer and Torres \cite{PFCT1, PFCT2} considered positive solution of nonlinear Schr\"odinger equation with non-local regional diffusion
\begin{equation}\label{03}
\epsilon^{2\alpha}(-\Delta)_{\rho}^{\alpha} u + u = f(u),\;\;u\in H^{\alpha}(\mathbb{R}^n),
\end{equation}
where the operator $(-\Delta)_{\rho}^{\alpha}$ is defined as above. Under suitable assumptions on the non-linearity $f$ and the range of scope $\rho$, they obtained the existence of a ground state by mountain pass argument and a comparison method. Furthermore, they analyzed symmetry properties and concentration phenomena of these solutions. These regional operators present various interesting characteristics that make them very attractive from the point of view of mathematical theory of non-local operators. We also mention the recent works by Torres \cite{CT1, CT2, CT3}, where existence, multiplicity and symmetry results are considered in bounded domain and $\mathbb{R}^n$.

We recall that when $(-\Delta)_{\rho}^{\alpha}$ is replaced by $(-\Delta)^{\alpha}$, Chen and Zheng \cite{GCYZ} studied (\ref{03}) with external potential $V(x)$ and $f(u) = |u|^{p-1}u$. They showed that when $n = 1,2,3$, $\epsilon$ is sufficiently small, $\max\{\frac{1}{2}, \frac{n}{4}\} < \alpha < 1$ and $V$ satisfies some smoothness and boundedness assumptions, equation (\ref{03}) has a nontrivial solution $u_\epsilon$ concentrated to some single point as $\epsilon \to 0$. Very recently, in \cite{JDMPJW},  D\'avila, del Pino and Wei generalized various existence results known for (\ref{03}) with $\alpha=1$ to the case of fractional Laplacian. Moreover, we also mention the works by Shang and Zhang \cite{XSJZ1, XSJZ2}, where it was considered the nonlinear fractional Schr\"odinger equation with competing potentials
\begin{equation}\label{04}
\epsilon^{2\alpha}(-\Delta)^{\alpha}u + V(x)u = K(x)|u|^{p-2}u + Q(x)|u|^{q-2}u,\;\;x\in \mathbb{R}^n,
\end{equation}
where $2<q<p<2_{\alpha}^{*}$. By using perturbative variational method, mountain pass arguments and Nehari manifold method, they analyzed the existence, multiplicity and concentration phenomena for the solutions of (\ref{04}).     

Motivated by these previous works, in this paper, our goal is to study the existence and concentration phenomena for the solutions of $(P)$. As pointed out in \cite{CASS, CASSJY, XWBZ}, the geometry of the ground state energy function $C(\xi)$, which is defined to be the ground state level associated with
$$
(-\Delta)^{\alpha}u + Q(\xi)u = K(\xi)|u|^{p-1}u,\;\;x\in \mathbb{R}^n,
$$
where $\xi \in \mathbb{R}^n$ is regard as a parameter instead of an independent variable, it is crucial in our approach. Here, the functions $\rho, Q$ and $K$ satisfy the following conditions: 
\begin{itemize}
	\item[$(H_0)$] There are positive real numbers $Q_\infty, K_\infty$ such that
	$$ 
	Q_\infty=\lim_{|\xi|\to +\infty}Q(\xi) \quad \mbox{and} \quad K_\infty=\lim_{|\xi|\to +\infty}K(\xi). 
	$$
	\item[$(H_{1})$] There are numbers $0<\rho_0<\rho_\infty\le \infty$ such that
	$$\rho_{0 } \leq \rho (\xi) < \rho_{\infty}, \quad \forall \xi \in \mathbb{R}^{n} \quad\mbox{and}\quad
	\lim_{|\xi| \to \infty} \rho (\xi) =\rho_\infty.
	$$
	\item[$(H_2)$] $Q,K: \mathbb{R}^n \to \mathbb{R}$ are continuous function satisfying
	$$
	0< a_1 \leq Q(\xi),K(\xi) \leq a_2 \quad \forall \xi \in \mathbb{R}^n 
	$$
	for some positive constants $a_1,a_2$.
\end{itemize}

\begin{Thm} \label{T1} Assume $(H_0)-(H_2)$. Then, if 
$$
\inf_{\xi \in \mathbb{R}^n}C(\xi) < \liminf_{|\xi| \to +\infty}C(\xi),  \leqno{(C)}
$$
problem $(P)$ has a ground state solution $u_\epsilon \in X^\epsilon$ for $\epsilon$ small enough. Moreover, for each sequence $\epsilon_m \to 0$, there is a subsequence such that for each $m \in \mathbb{N}$, the solution $u_{\epsilon_m}$ concentrates around a minimum point $\xi^*$ of the function $C(\xi)$, in the following sense: given $\delta>0$, there are $\epsilon_0, R>0$ such that
$$
\int_{B^{c}(\xi^*,\epsilon_m R)}|u_{\epsilon_m}|^{2}\,dx \leq \epsilon_m^{n}\delta \quad \mbox{and} \quad \int_{B(\xi^*,\epsilon_m R)}|u_{\epsilon_m}|^{2}\,dx \geq \epsilon_m^{n} C, \quad \forall \epsilon_m \leq \epsilon_0,
$$ 
where  $C$ is a constant independent of $\delta$ and $m$.  
\end{Thm}

We would like to point out that the condition $(C)$ is not empty, because it holds by supposing that there is $\xi_0 \in \R^n$ such that   
$$
\frac{Q(\xi_0)^{\frac{p+1}{p-1}-\frac{n}{2\alpha}}}{K(\xi_0)^{\frac{2}{p-1}}} < \frac{Q_\infty^{\frac{p+1}{p-1}-\frac{n}{2\alpha}}}{K_\infty^{\frac{2}{p-1}}}.
$$
For more details, see Corollary \ref{ntakey} in Section 3.


\section{Preliminaries}

In this section we recall some basic facts about the Sobolev space $H^\alpha(\mathbb{R}^n)$ such as embeddings and compactness properties. To begin with, we recall the following embedding theorem.
\begin{Thm} {\rm (\cite{EDNGPEV})}\label{FStm1}
Let $\alpha \in (0,1)$, then there exists a positive constant $C = C(n,\alpha)$ such that
\begin{equation}\label{P01}
\|u\|_{L^{2_{\alpha}^{*}}(\mathbb{R}^{n})}^{2} \leq C \int_{\mathbb{R}^{n}} \int_{\mathbb{R}^{n}} \frac{|u(x) - u(y)|^{2}}{|x-y|^{n+2\alpha}}dydx
\end{equation}
and then  $H^{\alpha}(\mathbb{R}^{n}) \hookrightarrow L^{q}(\mathbb{R}^{n})$ is continuous for all $q \in [2, 2_{\alpha}^{*}]$.
Moreover, $H^{\alpha}(\mathbb{R}^{n}) \hookrightarrow L^{q}(\Omega)$ is compact for any bounded set $\Omega\subset \mathbb{R}^{n}$ and for all $q \in [2, 2_{\alpha}^{*})$, where $2_{\alpha}^{*} = \frac{2n}{n-2\alpha}$ is the  critical exponent.
\end{Thm}

The next lemma establishes that $\|\,\,\,\|_{\rho_\epsilon}$ is equivalent to usual norm in $H^{\alpha}(\R^n)$. 
\begin{Prop}\label{FSprop1}
Suppose that ($H_{1}$) and $(H_2)$ hold  and set
$$
\|u\|=\left(\int_{\mathbb{R}^{n}}\int_{\R^n}\frac{|u(x) - u(z)|^{2}}{|x-z|^{n+2\alpha}}dzdx + \int_{\mathbb{R}^{n}}|u(x)|^{2}dx \right)^{\frac{1}{2}}
$$
the usual norm in $H^{\alpha}(\R^n)$. Then, there exists a constant $\mathfrak{S}> 0$ independent of $\epsilon$ such that
$$
\|u\| \leq \mathfrak{S} \|u\|_{\rho_\epsilon}, \quad \forall u \in H^{\alpha}(\R^n).
$$
From this, $\|\cdot\|$ and   $\|\cdot\|_{\rho_\epsilon}$ are equivalents norms in  $H^{\alpha}(\mathbb{R}^{n})$. 
\end{Prop}

\begin{proof}
Without loss of generality we will consider $\epsilon =1$. For $u \in  X^1=H^{\alpha}(\mathbb{R}^{n})$, the Fubini's Theorem together with $(H_1)$ and $(H_2)$ gives
\begin{equation}\label{P02}
\begin{aligned}
&a_1 \|u\|^2 = a_1 \int_{\mathbb{R}^{n}} |u(x)|^{2}dx + a_1 \int_{\mathbb{R}^{n}}\int_{B(x,\rho_{0})} \frac{|u(x) - u(z)|^{2}}{|x-z|^{n+2\alpha}}dzdx + \\
&\;\;\;\;\;\;\;\;\;\;\;\;\;\;\;\;\;a_1\int_{\mathbb{R}^{n}}\int_{B^c(x,\rho_{0})} \frac{|u(x) - u(z)|^{2}}{|x-z|^{n+2\alpha}}dzdx\\
&\leq \left(1+ \frac{2|S^{n-1}|}{\alpha \rho_{0}^{2\alpha}}\right) \int_{\mathbb{R}^n} Q(x)|u(x)|^{2}dx + a_1\int_{\mathbb{R}^{n}}\int_{B(x,\rho_{0})} \frac{|u(x) - u(z)|^{2}}{|x-z|^{n+2\alpha}}dzdx\\
&\leq A\left( \int_{\mathbb{R}^n} Q(x)|u(x)|^2dx + \int_{\mathbb{R}^n} \int_{B(0, \rho(x))}\frac{|u(x+z) - u(x)|^2}{|z|^{n+2\alpha}}dzdx\right),
\end{aligned}
\end{equation}
where $A=\max\left\{a_1, \left( 1 + \frac{2|S^{n-1}|}{\alpha \rho_{0}^{2\alpha}} \right)\right\}$. The proposition follows by taking $\mathfrak{S} = \frac{1}{a_1}A$. 
\end{proof}

The following lemma is a version of the concentration compactness principle proved by Felmer and Torres \cite{PFCT1}, which will be use later on.
\begin{Lem}\label{FSlem1}
Let $n\geq 2$. Assume that $\{u_{k}\}$ is bounded in $H_{\rho}^{\alpha}(\mathbb{R}^{n})$ and it satisfies
\begin{equation}\label{P03}
\lim_{k\to \infty} \sup_{y\in \mathbb{R}^{n}}\int_{B(y,R)}|u_{k}(x)|^{2}dx = 0,
\end{equation}
where $R>0$. Then $u_{k} \to 0$ in $L^{q}(\mathbb{R}^{n})$ for $2 < q < 2_{\alpha}^{*}$.
\end{Lem}


\section{Ground state}

We prove the existence of weak solution of $(P')$ finding a critical point of  the functional $I_{\rho_\epsilon}$. Using the embeddings given in Theorem \ref{FStm1}, it follows that  the functional $I_{\rho_\epsilon}$ is of class $C^{1}(X^\epsilon, \mathbb{R})$ with
$$
I'_{\rho_\epsilon}(u)v = \langle u,v \rangle_{\rho_\epsilon} - \int_{\mathbb{R}^{n}}K(\epsilon x)|u(x)|^{p-1}u(x)v(x)dx,\;\;\forall \;v\in X^{\epsilon}
$$
where
$$
\langle u,v \rangle_{\rho_\epsilon} =\int_{\mathbb{R}^n}\int_{B(0, \rho_\epsilon (x))} \frac{[u(x+z) - u(x)][v(x+z) - v(x)]}{|z|^{n+2\alpha}}dzdx +  \int_{\mathbb{R}^n}Q(\epsilon x)uv dx .
$$

Using well known arguments, it follows that $I_{\rho_\epsilon}$ verifies the mountain pass geometry. Then, there is a $(PS)_{c}$ sequence $\{u_k\} \subset X^{\epsilon}$ such that 
\begin{equation}\label{MPC1}
I_{\rho_\epsilon}(u_k) \to C_{\rho_\epsilon} \quad \mbox{and}\quad I'_{\rho_\epsilon} (u_k) \to 0
\end{equation}
where $C_{\rho_\epsilon}$ is the mountain pass level given by 
$$
C_{\rho_\epsilon} = \inf_{\gamma \in \Gamma_{\rho_{\epsilon}} \sup_{t\in [0,1]}} I_{\rho_\epsilon} (\gamma (t))>0
$$
with 
$$
\Gamma_{\rho_\epsilon} = \{ \gamma \in C([0,1], X^{\epsilon}): \gamma(0) = 0,\;\;I_{\rho_\epsilon} (\gamma (1))<0\}.
$$ 
In the sequel, $\mathcal{N}_{\rho_\epsilon}$ denotes the Nehari manifold associated  to the functional $I_{\rho_\epsilon}$, that is, 
$$
\mathcal{N}_{\rho_\epsilon} = \{u\in X^{\epsilon} \backslash \{0\}:\;\; I'_{\rho_\epsilon}(u)u=0\}.
$$
It is easy to see that all non trivial solutions of $(P')$ belongs to $ \mathcal{N}_{\rho_\epsilon}$. Moreover, by using standard arguments, it is possible to prove that 
\begin{equation} \label{ZZ0}
C_{\rho_\epsilon}=\inf_{u \in \mathcal{N}_{\rho_\epsilon}}I_{\rho_\epsilon}(u)
\end{equation}
and there is $\beta>0$ independent of $\epsilon$, such that
\begin{equation}\label{BETA1}
\beta \leq \|u\|_{\rho_\epsilon}^{2}, \quad \forall u \in X^{\epsilon}
\end{equation}
and so,
\begin{equation} \label{BETA2}
\beta \leq C_{\rho_\epsilon}, \quad \forall \epsilon >0.
\end{equation}

From (\ref{ZZ0}), if $C_{\rho_\epsilon}$ is a critical value of $I_{\rho_\epsilon}$ then it is the least energy critical value of $I_{\rho_\epsilon}$. Hereafter, we say that $C_{\rho_\epsilon}$ is the {\it ground state level} of $I_{\rho_\epsilon}$.


Now, we consider the following equation
\begin{equation}\label{05}
(-\Delta)^{\alpha} u + Q(\xi) u = K(\xi)|u|^{p-1}u, \;\;x\in \mathbb{R}^n,
\end{equation}
where $\xi \in \mathbb{R}^n$ is regard as a parameter instead of an independent variable. We define the energy functional  $J_\xi : H^{\alpha}(\R^n) \to \R$ associated with (\ref{05}) by 
\begin{equation}\label{n02}
\begin{array}{l}
J_\xi (u) =\displaystyle \frac{1}{2}\left( \int_{\mathbb{R}^n}\int_{\mathbb{R}^n} \frac{|u(x+z) - u(x)|^2}{|z|^{n+2\alpha}}dzdx + \int_{\mathbb{R}^n} Q(\xi)|u(x)|^2dx \right) \\
\mbox{}\\
\;\;\;\;\;\;\;\;\;\;\;\;\; - \displaystyle \frac{1}{p+1}\int_{\mathbb{R}^n} K(\xi)|u(x)|^{p+1}dx.
\end{array}.
\end{equation} 
Let 
$$
C(\xi) = \inf_{u \in \mathcal{N}_{\xi}} J_{\xi}(u)
$$
the ground state energy  associated with (\ref{05}), where $\mathcal{N}_\xi$ is the Nehari manifold defined as
$$
\mathcal{N}_\xi = \{u\in H^{\alpha}(\mathbb{R}^n)\setminus \{0\}: J'_{\xi}(u) u = 0\}.
$$

Arguing as above, we see that $C(\xi)>0$ and 
$$
C(\xi) = \inf_{v\in H^{\alpha}(\mathbb{R}^n)\setminus \{0\}} \max_{t>0} J_\xi (tv) = \inf_{\gamma \in \Gamma_\xi} \max_{t\in [0,1]} J_\xi (\gamma (t)),
$$
where 
$$
\Gamma_\xi =\{\gamma \in C([0,1], H^{\alpha}(\mathbb{R}^n)): \;\;\gamma (0) = 0, \;\;J_{\xi}(\gamma (1)) <0\}.
$$

By \cite{PFAQJT}, we know that for each $\xi \in \mathbb{R}^n$, problem (\ref{05}) has a nontrivial nonnegative ground state solution. Thus, $C(\xi)$ is the least critical value of $J_{\xi}$. Next, we will study the continuity of $C(\xi)$.

\begin{Lem}\label{lema1}
The function $\xi \to C(\xi)$ is continuous. 
\end{Lem}

\begin{proof}
Set $\{\xi_r \} \subset \mathbb{R}^n$ and $\xi_0 \in \mathbb{R}^n$ with
$$
\xi_r \to \xi_0 \quad \mbox{in} \quad \mathbb{R}^n.
$$
By using the conditions on $\rho$, $Q$ and $K$, we know that
$$
\liminf_{\xi \in \mathbb{R}^n}C(\xi)>0 \quad \mbox{and} \quad \limsup_{\xi \in \mathbb{R}^n}C(\xi)<+\infty.
$$
Next, we denote by $v_r \in H^{\alpha}(\mathbb{R}^n)$ the function which satisfies 
$$
J_{\xi_r}(v_r)=C(\xi_r) \quad \mbox{and} \quad  J'_{\xi_r}(v_r)=0. 
$$
In the sequel, we will consider two sequences $\{\xi_{r_j}\}$ and $\{\xi_{r_k}\}$ such that
$$
C(\xi_{r_j}) \geq C(\xi_0) \quad \forall r_j \eqno{(I)}
$$
and
$$
C(\xi_{r_k}) \leq C(\xi_0) \quad \forall r_k. \eqno{(II)}
$$
\noindent {\bf Analysis of $(I)$:} From the above commentaries, we know that $\{C_(\xi_{r_j})\}$ is bounded. Therefore, there are a subsequence $\{\xi_{{r_j}_i}\} \subset \{\xi_{r_j}\}$ and $C_0>0$ such that  
$$
C(\xi_{{r_j}_{i}}) \to C_0. 
$$
In the sequel, we will use the following notations:
$$
v_i=v_{{r_j}_i} \quad \mbox{and} \quad \xi_i=\xi_{{r_j}_i}.
$$
Thereby,
$$
\xi_i \to \xi_0 \quad \mbox{and} \quad C(\xi_i) \to C_0.
$$
\noindent {\bf Claim A:} \, $C_0=C(\xi_0)$. From (I), 
$$
\lim_{i}C(\xi_i) \geq C(\xi_0)
$$
and so,
\begin{equation} \label{C0}
C_0 \geq C(\xi_0).
\end{equation}
In the sequel, we set $w_0 \in H^{\alpha}(\mathbb{R}^n)$ be a function satisfying 
$$
J_{\xi_0}(w_0)=C(\xi_0) \quad \mbox{and} \quad J'_{\xi_0}(w_0)=0.
$$
Moreover, we denote by $t_{i}>0$ the real number which verifies  
$$
J_{\xi_i}(t_iw_0)=\max_{t \geq 0}J_{\xi_i}(t_iw_0).
$$
Thus, by definition of $C(\xi_0)$, 
$$
C(\xi_i)\leq J_{\xi_i}(t_iw_0).
$$
It is possible to prove that $\{t_i\}$ is a bounded sequence, then without lost of generality we can assume that $t_i \to t_0$. Now, by using the fact that the functions $\rho$ and $K$ are continuous, the Lebesgue's Theorem gives
$$
\lim_{i} J_{\xi_i}(t_iw_0)=J_{\xi_0}(t_0w_0) \leq J_{\xi_0}(w_0)=C(\xi_0),
$$
leading to
\begin{equation}\label{C1}
C_0 \leq C(\xi_0).
\end{equation}
From (\ref{C0})-(\ref{C1}),
$$
C(\xi_0)=C_0.
$$
The above study implies that
$$
\lim_{i}C(\xi_{{r_j}_i})=C(\xi_0).
$$
\noindent {\bf Analysis of $(II)$:} \, By using the definition of $\{v_r\}$, it is easy to prove that $\{v_r\}$ is a bounded sequence in  $H^{\alpha}(\mathbb{R}^n)$. Consequently, there is $v_0 \in H^{\alpha}(\mathbb{R}^n)$ such that
$$
v_r \rightharpoonup v_0 \quad \mbox{in} \quad H^{\alpha}(\mathbb{R}^n).
$$
By using Lemma \ref{FSlem1}, we can assume that $v_0 \not=0$, because for any translation of the type $\tilde{v}_n(x)=v_n(x+y_n)$ also satisfies
$$
J_{\xi_r}(\tilde{v}_r)=C(\xi_r) \quad \mbox{and} \quad J'_{\xi_r}(\tilde{v}_r)=0.
$$
The above information permits to conclude that $v_0$ is a nontrivial solution of the problem
\begin{eqnarray}\label{Eq0101}
(-\Delta)^{\alpha}u + Q(\xi_0)u = K(\xi_0)|u|^{p-1}u  \mbox{ in } \mathbb{R}^{n}, \;\;
u \in H^{\alpha}(\mathbb{R}^{n}).
\end{eqnarray}
By Fatous' lemma, it is possible to prove that 
\begin{equation} \label{C6}
\liminf_{r}J_{\xi_r}(v_r) \geq J_{\xi_0}(v_0).
\end{equation}
On the other hand, there is $s_r >0$ such that
$$
C(\xi_r) \leq J_{\xi_r}(s_r v_0) \quad \forall r.
$$
So
\begin{equation} \label{C7}
\limsup_{r}J_{\xi_r}(v_r)=\limsup_{r}C(\xi_r) \leq \limsup_{r}J_{\xi_r}(s_r v_0)=J_{\xi_0}(v_0).
\end{equation}
From (\ref{C6})-(\ref{C7}), 
$$
\lim_{r}J_{\xi_n}(v_n)=J_{\xi_0}(v_0).
$$
The last limit yields
$$
v_r \to v_0 \quad \mbox{in} \quad H^{\alpha}(\mathbb{R}^n).
$$
Since $\{C(\xi_{{r_j}_k})\}$ is bounded, there are a subsequence $\{\xi_{{r_j}_k}\} \subset \{\xi_{r_j}\}$ and $C_*>0$ such that  
$$
C(\xi_{{r_j}_{k}}) \to C_*. 
$$
In the sequel, we will use the following notations:
$$
v_k=v_{{r_j}_k} \quad \mbox{and} \quad \xi_k=\xi_{{r_j}_k}.
$$
Thus,
$$
v_k \to v_0, \quad \xi_k \to \xi_0 \quad \mbox{and} \quad C(\xi_k) \to C_*.
$$

In what follows, we denote by $t_{k}>0$ the real number which verifies  
$$
J_{\xi_0}(t_kv_k)=\max_{t \geq 0}J_{\xi_0}(tv_k).
$$
Thus, by definition of $C(\xi_0)$, 
$$
C(\xi_0)\leq J_{\xi_0}(t_kv_k).
$$
It is possible to prove that $\{t_k\}$ is a bounded sequence, then without lost of generality we can assume that $t_k \to t_*$. Now, by using the fact that the functions $\rho$ and $K$ are continuous, the Lebesgue's Theorem gives
$$
\lim_{k} J_{\xi_0}(t_kv_k)=J_{\xi_0}(t_*v_0)= \lim_{k} J_{\xi_k}(t_kv_k) \leq \lim_{k}C(\xi_k)=C_*.
$$
Thereby,
\begin{equation} \label{C3}
C(\xi_0)\leq C_*.
\end{equation}
On the other hand,  from $(II)$, 
$$
\lim_{k}C(\xi_k) \leq C(\xi_0)
$$
leading to 
\begin{equation} \label{C4}
C_* \geq C(\xi_0).
\end{equation}
From (\ref{C3})-(\ref{C4}), 
$$
C_* = C(\xi_0).
$$
The above study implies that
$$
\lim_{k}C(\xi_{{n_j}_k})=C(\xi_0).
$$
From $(I)$ and $(II)$,
$$
\lim_{r}(\xi_r)=C(\xi_0),
$$
showing the lemma.
\end{proof}



In what follows, we denote by  $D$ the ground state level of the function $J: H^{\alpha}(\R^n) \to \mathbb{R}$ given by
$$
J(u) = \frac{1}{2}\left( \int_{\mathbb{R}^n}\int_{\mathbb{R}^n} \frac{|u(x) - u(z)|^2}{|x-z|^{n+2\alpha}}dz dx + \int_{\mathbb{R}^n} |u(x)|^{2}dx\right) - \frac{1}{p+1}\int_{\mathbb{R}^n}|u|^{p+1}dx
$$	

Using the above notations, we have the following lemma

\begin{Lem}\label{lema2} The functions $C(\xi)$ verifies the following relation
\begin{equation}\label{g02}
C(\xi) = \frac{Q(\xi)^{\frac{p+1}{p-1} - \frac{n}{2\alpha}}}{K(\xi)^{\frac{2}{p-1}}}D, \quad \forall \xi \in\R^n.
\end{equation}
\end{Lem}
\begin{proof}
Let $u \in H^{\alpha}(\R^n)$ be a function verifying 
$$
J(u) = D \quad \mbox{and} \quad J'(u) = 0.
$$
For each $\xi \in \R^n$ fixed, let $\sigma^{2\alpha} = \frac{1}{Q(\xi)}$ and define 
$$
w(x) = \left[ \frac{Q(\xi)}{K(\xi)} \right]^{\frac{1}{p-1}}u(\frac{x}{\sigma}).
$$
Then, doing the change of variable $x = \sigma \tilde{x}$ and $z = \sigma \tilde{z}$ we obtain
$$
\begin{aligned}
&J_\xi (w) =\frac{1}{2}\left(\int_{\mathbb{R}^n}\int_{\mathbb{R}^n}\frac{|w(x+z) - w(x)|^{2}}{|z|^{n+2\alpha}}dz dx + \int_{\mathbb{R}^n}Q(\xi)w^2dx \right) - \frac{1}{p+1}\int_{\mathbb{R}^n}K(\xi)|w|^{p+1}dx\\
&= \frac{Q(\xi)}{2}\left( \sigma^{2\alpha}\int_{\mathbb{R}^n}\int_{\mathbb{R}^n}\frac{|w(x+z) - w(x)|^2}{|z|^{n+2\alpha}}dz dx + \int_{\mathbb{R}^n}w^2(x)dx\right) - \frac{1}{p+1}\int_{\mathbb{R}^n}K(\xi)|w|^{p+1}dx \\
&= \frac{Q(\xi)^{\frac{p+1}{p-1}}}{K(\xi)^{\frac{2}{p-1}}}\left[\left( \frac{\sigma^{2\alpha}}{2}\int_{\mathbb{R}^n}\int_{\mathbb{R}^n}\frac{|u(\frac{x}{\sigma} + \frac{z}{\sigma}) - u(\frac{x}{\sigma})|^2}{|z|^{n+2\alpha}}dz dx + \frac{1}{2}\int_{\mathbb{R}^n} |u(\frac{x}{\sigma})|^{2}dx \right) - \frac{1}{p+1}\int_{\mathbb{R}^n}|u(\frac{x}{\sigma})|^{p+1}dx\right]\\
&= \frac{Q(\xi)^{\frac{p+1}{p-1} - \frac{n}{2\alpha}}}{K(\xi)^{\frac{2}{p-1}}} J(u).
\end{aligned}
$$
A similar argument also gives $J'_\xi(w)(w)=0$, from where it follows 
$$
C(\xi) \leq \frac{Q(\xi)^{\frac{p+1}{p-1} - \frac{n}{2\alpha}}}{K(\xi)^{\frac{2}{p-1}}}D, \quad \forall \xi \in\R^n.
$$
The reverse inequality is obtained of the same way, finishing the proof.

\end{proof}

As a byproduct of the last proof, we have the following corollary

\begin{Cor}\label{ntakey}
By Lemma \ref{lema2}, if there is $\xi_0 \in \mathbb{R}^n$ such that
$$
\frac{Q(\xi_0)^{\frac{p+1}{p-1}-\frac{n}{2\alpha}}}{K(\xi_0)^{\frac{2}{p-1}}} < \frac{Q_\infty^{\frac{p+1}{p-1}-\frac{n}{2\alpha}}}{K_\infty^{\frac{2}{p-1}}},
$$
we have
$$
\inf_{\xi \in \mathbb{R}^n}C(\xi) < \liminf_{|\xi| \to +\infty}C(\xi) = C(\infty),
$$
where $C(\infty)$ is the mountain pass level of the functionals $J_\infty:H^{\alpha}(\mathbb{R}^n) \to \mathbb{R}$ given by
$$
J_\infty(u) = \frac{1}{2}\left( \int_{\mathbb{R}^n}\int_{\mathbb{R}^n} \frac{|u(x) - u(z)|^2}{|x-z|^{n+2\alpha}}dz dx + \int_{\mathbb{R}^n} Q_\infty |u|^2dx\right) - \frac{1}{p+1}\int_{\mathbb{R}^n}K_\infty |u|^{p+1}dx.
$$
\end{Cor}

\vspace{0.2 cm}

The next lemma studies the behavior of function $C_{\rho_{\epsilon}}(\xi)$ when $\epsilon$ goes to 0. 

\begin{Lem} \label{LIMITC} 
$\displaystyle \limsup_{\epsilon \to 0}C_{\rho_{\epsilon}} \leq  \inf_{\xi \in \R^n}C(\xi)$. Hence, $\displaystyle \limsup_{\epsilon \to 0}C_{\rho_{\epsilon}} <  C(\infty)$.	
\end{Lem}
\begin{proof}
 Fix $\xi_0 \in \mathbb{R}^N$ and $w\in H^{\alpha}(\mathbb{R}^n)$ with
	$$
	J_{\xi_0}(w) = \max_{t\geq 0} J_{\xi_0}(tw)=C(\xi_0) \quad \mbox{and} \quad 	J_{\xi_0}'(w)=0
	$$
where
	$$
	J_{\xi_0}(u) = \frac{1}{2}\left( \int_{\mathbb{R}^n}\int_{\mathbb{R}^n} \frac{|u(x) - u(z)|^2}{|x-z|^{n+2\alpha}}dz dx + \int_{\mathbb{R}^n} Q(\xi_0)|u(x)|^{2}dx\right) - 
	\frac{1}{p+1}\int_{\mathbb{R}^n}K(\xi_0)|u|^{p+1}dx.
	$$	
	Then, we take $w_\epsilon (x) = w(x - \frac{\xi_0}{\epsilon})$ and $t_\epsilon >0$ satisfying 
	$$
	C_{\rho_\epsilon} \leq I_{\rho_\epsilon}(t_\epsilon w_\epsilon) = \max_{t\geq 0} I_{\rho_\epsilon}(tw_\epsilon).
	$$
	The change of variable $\tilde{x} = x - \frac{\xi_0}{\epsilon}$ gives   
	$$
	\begin{aligned}
	I_{\rho_\epsilon}(t_\epsilon w_\epsilon) &= \frac{t_{\epsilon}^{2}}{2}\left( \int_{\mathbb{R}^n}\int_{B(0, \frac{1}{\epsilon}\rho (\epsilon x))}\frac{|w_\epsilon(x+z) - w_\epsilon (x)|^2}{|z|^{n+2\alpha}}dxdx + \int_{\mathbb{R}^n}Q(\epsilon x)w_\epsilon^2(x)dx \right)\\
	&-\frac{t_{\epsilon}^{p+1}}{p+1}\int_{\mathbb{R}^n}K(\epsilon x)w_{\epsilon}^{p+1}(x)dx\\
	&= \frac{t_{\epsilon}^{2}}{2}\left( \int_{\mathbb{R}^n}\int_{B(0, \frac{1}{\epsilon}\rho (\epsilon \tilde{x} + \xi_0))} \frac{|w(\tilde{x} + z) - w(\tilde{x})|^2}{|z|^{n+2\alpha}}dz d\tilde{x} + \int_{\mathbb{R}^n} Q(\epsilon \tilde{x} + \xi_0)w^2(\tilde{x})d\tilde{x}\right)\\
	&-\frac{t_{\epsilon}^{p+1}}{p+1}\int_{\mathbb{R}^n}K(\epsilon \tilde{x} + \xi_0)w^{p+1}(\tilde{x})d\tilde{x}.
	\end{aligned}
	$$
	Thereby, considering a sequence $\epsilon_n \to 0$, the fact that $I'_{\rho_{\epsilon_n}}(t_{\epsilon_n} w_{\epsilon_n})(t_{\epsilon_n} w_{\epsilon_n})=0$ yields $\{t_{\epsilon_n}\}$ is bounded. Thus, we can assume that 
	$$
	t_{\epsilon_n} \to t_*>0,
	$$
	for some $t_*>0$.  Using a change variable as above, we can infer that 
	$$
	J_{\xi_0}'(t_*w)(t_*w)=0.
	$$
	On the other hand, we know that $J_{\xi_0}'(w)(w)=0$. Then by uniqueness,  we must have 
	$$
	t_*=1.
	$$
	From this, 
	$$
	I_{\rho_{\epsilon_n}}(t_{\epsilon_n} w_{\epsilon_n}) \to J_{ \xi_0}(w)=C(\xi_0)\;\;\mbox{as}\;\;\epsilon \to 0.
	$$ 
As the point $\xi_0 \in \R^n$ is arbitrary, the lemma is proved.

\end{proof}

\begin{Thm}\label{main1}
For $\epsilon >0$ small enough, the problem $(P')$ has a positive least energy solution.
\end{Thm}

\begin{proof} In what follows, we denote by $\{u_k\} \subset H^{\alpha}(\mathbb{R}^N)$ a sequence satisfying
$$
I_{\rho_{\epsilon}}(u_k) \to C_{\rho_{\epsilon}} \quad \mbox{and} \quad I'_{\rho_{\epsilon}}(u_k) \to 0.
$$	
If $u_k \rightharpoonup 0$ in $H^{\alpha}(\mathbb{R}^N)$, then 
\begin{equation}\label{lim0}
u_k \to 0\;\;\mbox{ in}\;\;L_{loc}^{p}(\mathbb{R}^n)\;\;\mbox{for}\;\; p\in [2, 2_{\alpha}^{*}).
\end{equation} 
By $(H_0)$, we can take $\delta, R>0$ such that
\begin{equation}\label{eq17}
Q_\infty - \delta \leq Q(x)\leq Q_\infty + \delta\;\;\mbox{and}\;\;K_\infty - \delta \leq K(x) \leq K_\infty + \delta
\end{equation}
for all $|x|\geq R$. Then, for all $t\geq 0$,
$$
\begin{aligned}
I_{\rho_\epsilon} (tu_k) & = I_{\epsilon, \infty}^{\delta} (tu_k)  + \frac{t^2}{2} \int_{\mathbb{R}^n} [Q(x) - Q_\infty + \delta] |u_k(x)|^2dx \\
&+ \frac{t^{p+1}}{p+1}\int_{\mathbb{R}^n} [K_\infty + \delta -K(x)]|u_k(x)|^{p+1}dx\\
&\geq I_{\epsilon , \infty}^{\delta}(tu_k) + \frac{t^2}{2} \int_{B(0, \frac{R}{\epsilon})} [Q(x) - Q_\infty + \delta] |u_k(x)|^2dx\\
& + \frac{t^{p+1}}{p+1}\int_{B(0, \frac{R}{\epsilon})} [K_\infty + \delta -K(x)]|u_k(x)|^{p+1}dx,
\end{aligned}
$$
where
$$
\begin{aligned}
I_{\epsilon, \infty}^{\delta} (u) &= \frac{1}{2}\left( \int_{\mathbb{R}^n}\int_{B(0, \frac{1}{\epsilon}\rho (\epsilon x)} \frac{|u(x+z) - u(x)|^2}{|z|^{n+2\alpha}}dxdx + \int_{\mathbb{R}^n}(Q_\infty - \delta)|u(x)|^2dx \right) \\&- \frac{1}{p+1}\int_{\mathbb{R}^n} (K_\infty + \delta)|u(x)|^{p+1}dx.
\end{aligned}
$$
Now we know that there exists a bounded sequence $\{\tau_k\}$ such that  
$$
I_{\epsilon, \infty}^{\delta}(\tau_k u_k) \geq C(\frac{\rho(\epsilon x)}{\epsilon}, Q_\infty - \delta, K_\infty + \delta), 
$$
where
$$
C(\frac{\rho(\epsilon x)}{\epsilon}, Q_\infty - \delta, K_\infty + \delta) = \inf_{v\in H^{\alpha}(\mathbb{R})\setminus \{0\}} \sup_{t\geq 0}I_{\epsilon , \infty}^{\delta}(tv)
$$
Thus, 
$$
\begin{aligned}
C_{\rho_{\epsilon}} &\geq C(\frac{\rho(\epsilon x)}{\epsilon}, Q_\infty - \delta, K_\infty + \delta)+ \frac{\tau_k^2}{2} \int_{B(0, \frac{R}{\epsilon})} [Q(x) - Q_\infty + \delta] |u_k(x)|^2dx \\
&+ \frac{\tau_k^{p+1}}{p+1}\int_{B(0, \frac{R}{\epsilon})} [K_\infty + \delta -K(x)]|u_k(x)|^{p+1}dx 
\end{aligned}
$$
Taking the limit as $k\to \infty$, and after $\delta \to 0$, we find
\begin{equation}\label{eq21}
c_{\rho_\epsilon} \geq C(\frac{\rho(\epsilon x)}{\epsilon}, Q_\infty, K_\infty)
\end{equation}
where $C(\frac{\rho(\epsilon x)}{\epsilon}, Q_\infty, K_\infty )$ denotes the mountain pass level of the functional
$$
\begin{array}{l}
I^{0}_{\infty,\xi}(u) = \displaystyle \frac{1}{2}\left( \int_{\mathbb{R}^n}\int_{B(0,\frac{1}{\epsilon}\rho(\epsilon x)} \frac{|u(x+z) - u(x)|^2}{|z|^{n+2\alpha}}dz dx + \int_{\mathbb{R}^n} Q_\infty |u|^2dx\right) - \\
\mbox{}\\
\;\;\;\;\;\;\;\;\;\;\;\;\; \displaystyle \frac{1}{p+1}\int_{\mathbb{R}^n}K_\infty u^{p+1}dx.
\end{array}
$$
A standard argument shows that
$$
\liminf_{\epsilon \to 0}C(\frac{\rho(\epsilon x)}{\epsilon}, Q_\infty, K_\infty) \geq C(\infty).
$$ 
Therefore, if there is $\epsilon_n \to 0$ such that the $(PS)_{C_{\rho_{\epsilon_n}}}$ sequence has weak limit equal to zero, we must have
$$
C_{\rho_{\epsilon_n}} \geq C(\frac{\rho(\epsilon_n x)}{\epsilon_n}, Q_\infty, K_\infty), \quad \forall n \in \mathbb{N},
$$
leading to 
$$
\liminf_{n \to +\infty} C_{\rho_{\epsilon_n}} \geq C(\infty),
$$
which contradicts Lemma \ref{LIMITC}.  This proves that the weak limit is non trivial for $\epsilon >0$ small enough and standard arguments show that its energy is equal to $C_{\rho_{\epsilon}}$, showing the desired result. 
\end{proof}

\section{Concentration of the solutions $u_\epsilon$}

\begin{Lem}\label{Clm3}
If $v_\epsilon$ is family solutions  of  $(P')$ with critical value $C_{\rho_\epsilon}$, then there exists a family $\{y_{\epsilon}\}$ and positive constants $R$ and $\beta$ such that
\begin{equation}\label{Ceq10}
\liminf_{\epsilon \to 0^{+}} \int_{B(y_{\epsilon}, R)}|v_{\epsilon}|^{2}\,dx \geq \beta >0.
\end{equation}
\end{Lem}

\begin{proof}
First we note that, by ($H_1$) and ($H_3$) we have
$$
\begin{aligned}
I_{\rho_\epsilon}(v) \geq I_{*}(v) &= \frac{1}{2}\left(\int_{\mathbb{R}^n}\int_{B(0, \rho_0)} \frac{|v(x+z) - v(x)|^2}{|z|^{n+2\alpha}}dz dx + \int_{\mathbb{R}^n}a_1|v|^2dx\right)\\
& - \frac{1}{p+1}\int_{\mathbb{R}^n}a_2 |u|^{p+1}dx.
\end{aligned}
$$ 
Let $\mathcal{N}_{*} = \{v\in H^{\alpha}(\mathbb{R}^n)\setminus \{0\}:\;\; I'_{*}(v)v =0\}$. Then, for each $v\in \mathcal{N}_{*}$ there exists unique $t_v>0$ such that $t_vv \in \mathcal{N}_*$. Hence,
\begin{equation}\label{c0}
\begin{aligned}
0< C(\rho_0, a_1,a_2) &= \inf_{v\in \mathcal{N}_*}I_*(v)  \leq \inf_{v\in \mathcal{N}_*}I_{\rho_\epsilon}(v)\\
&\leq \inf_{v\in \mathcal{N}_*}I_{\rho_{\epsilon}} (t_vv) = \inf_{u\in \mathcal{N}_{\rho_\epsilon}}I_{\rho_\epsilon}(u) = C_{\rho_\epsilon}.
\end{aligned}
\end{equation}
Now, by contradiction, if (\ref{Ceq10}) does not hold, then there exists a sequence $v_{k} = v_{\epsilon_{k}}$ such that
$$
\lim_{k\to \infty} \sup_{y \in \mathbb{R}^{n}} \int_{B(y ,R)} |v_{k}|^{2}dx = 0.
$$
By Lemma \ref{FSlem1}, 
$
v_{k} \to 0$ in $L^{q}(\mathbb{R}^{n})$ for any $2 < q < 2_{\alpha}^{*}.
$
However, this is impossible since by (\ref{c0}) 
$$
\begin{aligned}
0<C(\rho_0, a_1,a_2) \leq C_{\rho_\epsilon} &= I_{\rho_\epsilon}(v_\epsilon) - \frac{1}{2}I'_{\rho_\epsilon}(v_\epsilon)v_\epsilon\\
& = \frac{p-1}{2(p+1)}\int_{\mathbb{R}^{n}}K(\epsilon x)|v_{\epsilon}|^{p+1}dx \\
&\leq \frac{p-1}{2(p+1)}\int_{\mathbb{R}^n} a_2|v_\epsilon|^{p+1}dx \to 0,\;\;\mbox{as} \;\;k \to \infty.	
 \end{aligned}
$$
\end{proof}

Now let 
\begin{equation}\label{sol}
w_{\epsilon}(x) = v_{\epsilon}(x + y_{\epsilon}) = u_{\epsilon}(\epsilon x + \epsilon y_{\epsilon}),
\end{equation} 
then by (\ref{Ceq12}),
\begin{equation}\label{Ceq11}
\liminf_{\epsilon \to 0^{+}}\int_{B(0,R)} |w_{\epsilon}|^{2}dx \geq \beta > 0.
\end{equation}
To continue, we consider the rescaled scope function $\overline\rho_\epsilon$, defined as, 
$$
\bar\rho_\epsilon(x)=\frac{1}{\epsilon}\rho(\epsilon x+\epsilon y_\epsilon)
$$
and then $w_\epsilon$ satisfies the equation 
\begin{equation}\label{Ceq12}
(-\Delta)_{\overline{\rho}_{\epsilon}}^{\alpha}w_{\epsilon}(x) + Q(\epsilon x + \epsilon y_\epsilon)w_{\epsilon}(x) = K(\epsilon x + \epsilon y_\epsilon)|w_{\epsilon}(x)|^{p-1}w_\epsilon(x), \;\;\mbox{in}\;\;\mathbb{R}^{n}.
\end{equation}

\begin{Lem}\label{Clm4}
The sequence $\{\epsilon y_\epsilon\}$ is bounded. Moreover, if $\epsilon_m y_{\epsilon_m} \to \xi^*$, then
$$
C(\xi^*) = \inf_{\xi \in \mathbb{R}^n} C(\xi).
$$ 
\end{Lem}
\begin{proof}
Suppose by contradiction that $|\epsilon_m y_{\epsilon_m}| \to \infty$ and consider the function $w_{\epsilon_m}$ defined by (\ref{sol}), which satisfies (\ref{Ceq12}). Since $\{C_{\rho_{\epsilon_m}}\}$ is bounded, so the sequence $\{w_m\}$ is also bounded in $H^{\alpha}(\mathbb{R}^n)$. Then $w_m \rightharpoonup w$ in $H^\alpha (\mathbb{R}^n)$, and $w\neq 0$ by Lemma \ref{Clm3} . Now, by (\ref{Ceq12}) we get the following equality
$$
\begin{aligned}
&\int_{\mathbb{R}^n}\int_{B(0, \frac{1}{\epsilon_m}\rho (\epsilon_m x + \epsilon_m y_{\epsilon_m}))}\frac{[w_{m}(x+z) - w_{m}(x)][w(x+z) - w(x)]}{|z|^{n+2\alpha}}dz dx\\
& + \int_{\mathbb{R}^n}Q(\epsilon_m x + \epsilon_m y_{\epsilon_m})w_{m}wdx = \int_{\mathbb{R}^n}K(\epsilon_m x + \epsilon_m y_{\epsilon_m})|w_{m}|^{p-1}w_{m}w dx.
\end{aligned}
$$
So, by Fatou's Lemma we get
\begin{equation}\label{Ceq13}
\int_{\mathbb{R}^n}\int_{\mathbb{R}^n}\frac{|w(x+z) - w(x)|^2}{|z|^{n+2\alpha}}dz dx + \int_{\mathbb{R}^n}Q_\infty |w|^2dx \leq  \int_{\mathbb{R}^n}K_\infty|w|^{p+1}dx
\end{equation}
Let $\theta >0$ such that 
$$
J_\infty(\theta w) = \max_{t\geq 0} J_\infty (t w).
$$
From (\ref{Ceq13}), $\theta \in (0,1]$, whence 
$$ 
\begin{aligned}
C(\infty) &\leq J_\infty (\theta w) - \frac{1}{2}J'_{\infty}(\theta w)\theta w = \left( \frac{1}{2} - \frac{1}{p+1}\right) \theta^{p+1}\int_{\mathbb{R}^n} K_\infty|w(x)|^{p+1}dx\\
&\leq \left( \frac{1}{2} - \frac{1}{p+1} \right)\int_{\mathbb{R}^n}K_\infty|w(x)|^{p+1}dx\\
&\leq \left( \frac{1}{2} - \frac{1}{p+1} \right) \liminf_{m \to \infty} \int_{\mathbb{R}^n} K(\epsilon_m x + \epsilon_m y_{\epsilon_m})|w_m(x)|^{p+1}dx\\
&= \liminf_{m\to \infty} C_{{\rho}_{\epsilon_n}} < C(\infty)
\end{aligned}
$$
which is a contradiction. So $\{\epsilon y_\epsilon\}$ is bounded. Thus, there exists a subsequence of $\{\epsilon y_\epsilon\}$ such that $\epsilon_m y_{\epsilon_m} \to \xi^*$.

Repeating above arguments, define the function
$$
w_m(x)=v_{\epsilon_m}(x + y_{\epsilon_m}) = u_{\epsilon_m}(\epsilon_m x + \epsilon_m y_{\epsilon_m}).
$$ 
This function satisfies the equation (\ref{Ceq12}), and again $\{w_m\}$ is bounded in $H^{\alpha}(\mathbb{R}^n)$. Then $w_m \rightharpoonup w$ in $H^\alpha (\mathbb{R}^n)$, where $w$ satisfy the following equation
\begin{equation}\label{Ceq14}
(-\Delta)^{\alpha}w + Q(\xi^*)w = K(\xi^*)|w|^{p-1}w, \quad x \in \mathbb{R}^n,
\end{equation}
in the weak sense. Furthermore, associated to (\ref{Ceq14}) we have the energy functional 
$$
\begin{aligned}
J_{\xi^*}(u) &= \frac{1}{2}\left( \int_{\mathbb{R}^n}\int_{\mathbb{R}^n} \frac{|u(x+z) - u(x)|^2}{|z|^{n+2\alpha}}dz dx + \int_{\mathbb{R}^n} Q(\xi^*)|u(x)|^2dx\right)\\
&- \frac{1}{p+1}\int_{\mathbb{R}^n} K(\xi^*)|u(x)|^{p+1}dx.
\end{aligned}
$$
Using $w$ as a test function in (\ref{Ceq12}) and taking the limit of $m \to +\infty$, we get
$$
\int_{\mathbb{R}^n}\int_{\mathbb{R}^n}\frac{|w(x+z) - w(x)|^2}{|z|^{n+2\alpha}}dxdx + \int_{\mathbb{R}^n}Q(\xi^*)|w(x)|^2dx \leq \int_{\mathbb{R}^n}K(\xi^*)|w|^{p+1}dx,
$$
which implies that there exists $\theta \in (0, 1]$ such that
$$
J_{\xi^*}(\theta w) = \max_{t\geq 0}J_{\xi^*}(tw). 
$$
So, by Lemma \ref{LIMITC}, 
$$
\begin{aligned}
C(\xi^*)&\leq J_{\xi^*}(\theta w) = \left( \frac{1}{2} - \frac{1}{p+1} \right)\theta^{p+1} \int_{\mathbb{R}^n}K(\xi^*)|w(x)|^{p+1}dx\\
&\leq \left( \frac{1}{2}-\frac{1}{p+1} \right) \liminf_{m \to \infty} \int_{\mathbb{R}^n}K(\epsilon_m x + \epsilon_m y_{\epsilon_m})|w_{m}(x)|^{p+1}dx\\
&= \liminf_{m\to \infty} [I_{{\rho}_{\epsilon_m}}(v_{\epsilon_m}) - I'_{{\rho}_{\epsilon_m}}(v_{\epsilon_m})v_{\epsilon_m}]\\
&=\liminf_{m\to \infty} C_{{\rho}_{\epsilon_m}} \leq \limsup_{m\to \infty} C_{{\rho}_{\epsilon_m}} \leq \inf_{\xi\in \mathbb{R}^n}C(\xi),
\end{aligned}
$$
showing that $C(\xi^*)=\displaystyle \inf_{\xi\in \mathbb{R}^n}C(\xi)$. 
\end{proof}

Now we   prove  the convergence of $w_\epsilon$ as $\epsilon\to 0$.
\begin{Lem}\label{Clm5} 
For every sequence $\{\epsilon_m\}$ there is a subsequence, we keep calling the same, so
that 
$w_{\epsilon_m}=w_{m} \to w$ in $H^{\alpha}(\mathbb{R}^{n})$, when $m\to \infty$, where $w$ is a solution of \equ{Ceq14}.
\end{Lem}
\begin{proof}
Since $w$ is a solution of (\ref{Ceq14}), from Lemma \ref{LIMITC} , we have
$$
\begin{aligned}
&\inf_{\xi\in \mathbb{R}^n} C(\xi) = C(\xi^*) \leq J_{\xi^*}(w) = J_{\xi^*}(w) - \frac{1}{2}J'_{\xi^*}(w)w\\
& = \left(\frac{1}{2} - \frac{1}{p+1}\right) \int_{\mathbb{R}^n}K(\xi^*)|w|^{p+1}dx\\
&\leq \left(\frac{1}{2} - \frac{1}{p+1} \right)\liminf_{m\to \infty} \int_{\mathbb{R}^n} K(\epsilon_m x + \epsilon_m y_{\epsilon_m})|w_m|^{p+1}(x)dx\\
& \leq \left(\frac{1}{2} - \frac{1}{p+1} \right)\limsup_{m\to \infty}\int_{\mathbb{R}^n} K(\epsilon_m x + \epsilon_m y_{\epsilon_m})|w_m|^{p+1}dx\\
& = \left(\frac{1}{2} - \frac{1}{p+1} \right)\limsup_{m\to \infty}\int_{\mathbb{R}^n} K(\epsilon_m x )|v_m|^{p+1}dx\\
&\leq \limsup_{m\to \infty} \left( I_{{\rho}_{\epsilon_m}}(v_m) - \frac{1}{p+1}I'_{\rho_{\epsilon_m}}(v_m)v_m \right) \\
&= \limsup_{m\to \infty} C_{\overline{\rho}_{\epsilon_m}} \leq \inf_{\xi \in \mathbb{R}^n} C(\xi).
\end{aligned}
$$
The above estimates gives 
$$
\lim_{m\to \infty} \int_{\mathbb{R}^n}K(\epsilon_m x + \epsilon_m y_{\epsilon_m})|w_m|^{p+1}dx= \int_{\mathbb{R}^n}K(\xi^*)|w|^{p+1}dx.
$$
Consequently,  
$$
\begin{aligned}
(a)&\lim_{m\to \infty} \int_{\mathbb{R}^n}\int_{\mathbb{R}^n} \frac{|w_m(x+z) - w_m(x)|^2}{|z|^{n+2\alpha}}dz dx = \int_{\mathbb{R}^n} \int_{\mathbb{R}^n} \frac{|w(x+z)-w(x)|^2}{|z|^{n+2\alpha}}dz dx\\
(b)&\lim_{m\to \infty} \int_{\mathbb{R}^n} Q(\epsilon_m x + \epsilon_m y_{\epsilon_m})|w_m(x)|^{2}dx = \int_{\mathbb{R}^n} Q(\xi^*)|w(x)|^2dx.
\end{aligned}
$$
From $(b)$, given $\delta>0$ there exists $R>0$ such that
$$
\int_{|x|\geq R}Q(\epsilon_m x + \epsilon_m y_{\epsilon_m})|w_m(x)|^{2}dx \leq \delta.
$$
Furthermore, using $(H_3)$, we obtain
\begin{equation}\label{Ceq15}
\int_{|x|\geq R}|w_m(x)|^2dx \leq \frac{\delta}{a_1}.
\end{equation} 
On the other hand
\begin{equation}\label{Ceq16}
\lim_{m\to \infty}\int_{|x|\leq R}|w_m(x)|^2dx = \int_{|x|\leq R}|w(x)|^2dx.
\end{equation}
From (\ref{Ceq15}) and (\ref{Ceq16}), $w_m \to w $ in $L^2(\mathbb{R}^n)$. From this,  given $\delta>0$ there are $\epsilon_0, R>0$ such that
$$
\int_{B^{c}(x^*,\epsilon_m R)}|u_{\epsilon_m}|^{2}\,dx \leq \epsilon_m^{n}\delta \quad \mbox{and} \quad \int_{B(x^*,\epsilon_m R)}|u_{\epsilon_m}|^{2}\,dx \geq \epsilon_m^{n} C, \quad \forall \epsilon_m \leq \epsilon_0,
$$ 
where  $C$ is a constant independent of $\delta$ and $m$, showing the concentration of solutions $\{u_{\epsilon_{n}}\}$.
\end{proof}

\end{document}